\documentclass[11pt]{amsart}

\scrollmode

\usepackage{times}
\usepackage{amsmath,graphicx}
\usepackage{latexsym}
\usepackage{amscd,amssymb}
\usepackage[all]{xy}

\newtheorem{thm}{Theorem}[section]

\newtheorem{lem}[thm]{Lemma}
\newtheorem{cor}[thm]{Corollary}

\newtheorem{conj}[thm]{Conjecture}

\theoremstyle{definition}
\newtheorem{ex}[thm]{Example}

\theoremstyle{remark}
\newtheorem{rem}[thm]{Remark}

\numberwithin{equation}{section}

\title[Homology classes of negative square]{Homology classes of negative square and embedded surfaces in 4-manifolds}
\author{M.~J.~D.~Hamilton}
\address{      Institute for Geometry and Topology\\
               University of Stuttgart\\
               Pfaffenwaldring 57\\
               70569 Stuttgart\\
               Germany}
\email{mark.hamilton@math.lmu.de}
\date{\today}

\begin{document}

\begin{abstract} Let $X$ be a simply-connected closed oriented 4-manifold and $A$ an embedded surface of genus $g$ and negative self-intersection $-N$. We show that for fixed genus $g$ there is an upper bound on $N$ if the homology class of $A$ is divisible or characteristic. In particular, for genus zero, there is a lower bound on the self-intersection of embedded spheres in these kinds of homology classes. This question is related to a problem from the Kirby list.
\end{abstract}

\maketitle

\section{Introduction}

A well-known problem in the smooth topology of simply-connected closed oriented 4-manifolds is to determine for a given second integral homology class the minimal genus of a smoothly embedded closed connected oriented surface representing that class. This problem dates back to the work of Kervaire and Milnor \cite{KeMi} and has been studied a lot since the introduction of the Donaldson and Seiberg-Witten invariants. The problem is harder and less is known in the case of homology classes of negative self-intersection (if the manifold does not admit an orientation reversing self-diffeomorphism). This is partly due to the fact that the Seiberg-Witten invariants are sensitive to the orientation of the 4-manifold. If the self-intersection number gets very negative the adjunction inequality \cite{OzSz} from Seiberg-Witten theory reduces to the fact that the genus of the surface is at least zero. For example, the minimal genus problem for the $K3$ surface for classes of non-negative square has been  solved completely \cite{La}. On the other hand, not much is known for classes of negative square. Finashin and Mikhalkin \cite{FM} constructed for any negative even number $n\geq -86$ an embedded sphere in the $K3$ surface with this self-intersection number; see also \cite{Ak}. In general, Problem 4.105 from the Kirby list \cite{Ki} asks if there is a (negative) lower bound on the self-intersection of embedded spheres in a given 4-manifold $X$ (a 4-manifold with non-vanishing Seiberg-Witten invariants does not contain embedded spheres of positive self-intersection). One can also ask the following question: Given a genus $g$, is there a lower bound on the self-intersection number of embedded surfaces of genus $g$ in the 4-manifold $X$?

In this note, we will show that there is a lower bound on the self-intersection number if we fix the type of the homology class $A$ in following sense: We assume that the integral homology class $A$ is divisible by an integer greater than one or that it is characteristic. The proof in both cases uses branched coverings in a similar way to the paper \cite{KM} due to Kotschick and Mati\'c together with the solution of the $\frac{5}{4}$-conjecture due to Furuta \cite{F}. There is also another approach using the $G$-signature theorem \cite{HS,R} that leads to the best results for divisible classes in our situation. The result for divisible classes therefore is implicit in these papers, but it is a worthwhile preparation for the case of a characteristic class. In this case the other ingredient is a lemma that allows to represent a multiple of a given homology class by a surface of a certain genus. The argument here seems to work only with the $\frac{5}{4}$-theorem. 

Note that Kervaire and Milnor have shown that $A^2\equiv \sigma(X)\bmod 16$ if $A$ is a characteristic homology class represented by an embedded sphere, where $\sigma(X)$ denotes the signature of $X$. The following lower bound on $A^2$ is a consequence of Theorem \ref{thm char}.
\begin{cor}
Let $X$ be a simply-connected closed oriented 4-manifold and $A$ a characteristic class represented by an embedded sphere with $A^2<0$. Then
\begin{equation*}
A^2\geq -(4b_2(X)-5\sigma(X)-8).
\end{equation*}
\end{cor}
The spheres in the example of Finashin and Mikhalkin represent indivisible homology classes in the $K3$ surface. Unfortunately, the method presented here does not give a lower bound on the self-intersection number in this case. In the final section we indicate a strategy how to obtain a bound in the general situation, that depends on an open, yet not completely unrealistic, conjecture.

\section{Branched coverings and the $\frac{5}{4}$-theorem}\label{sect facts}

We first need some preparations. From now on, $X$ will denote a closed oriented simply-connected smooth 4-manifold whose second Betti number is non-zero, so that $X$ is not a homotopy sphere (the results in this paper remain valid if $H_1(X;\mathbb{Z})=0$). All embedded surfaces in the following will be connected, closed, oriented and smoothly embedded. Suppose that $B\subset X$ is such an embedded surface in $X$ of genus $g_B$, whose homology class $[B]\in H_2(X;\mathbb{Z})$ is divisible by a prime power $q=p^r$, where $p$ is a prime and $r\geq 1$ an integer. We can then consider the $q$-fold cyclic branched covering $\phi\colon Y\rightarrow X$, branched over $B$. The second Betti number and the signature $Y$ are known from \cite{Hirz,HS,KM,R}. They are given by:
\begin{align*}
b_2(Y)&=qb_2(X)+2(q-1)g_B\\
\sigma(Y)&=q\sigma(X)-\frac{q^2-1}{3q}B^2.
\end{align*}
The formula for the second Betti number follows from the well-known formula for the Euler characteristic of a branched covering and the fact that the first Betti number of $Y$ is zero. In addition it is known \cite{Na} that $Y$ is spin if and only if the class $(q-1)\frac{1}{q}[B]$ is characteristic, i.e.~the mod 2 reduction of its Poincar\'e dual is equal to the second Stiefel-Whitney class $w_2(X)$. Hence if $q$ is odd, then $Y$ is spin if and only if $X$ is spin. If $q$ is even, then $Y$ is spin if and only if $\frac{1}{q}[B]$ is characteristic. 

In any case, we have the inequality
\begin{equation*}
b_2(Y)\geq|\sigma(Y)|.
\end{equation*}
Suppose that $Y$ is a spin manifold. Under our assumptions the second Betti number of $Y$ is non-zero. In this situation Furuta's theorem \cite{F} on the $\frac{5}{4}$-conjecture applies and shows that 
\begin{equation*}
b_2(Y)\geq \frac{5}{4}|\sigma(Y)|+2.
\end{equation*}

\section{The case of divisible $A$}\label{sect div}

The following theorem can be deduced from results in \cite{HS,R} due to Hsiang-Szczarba and Rohlin.
\begin{thm}\label{thm div0}
Let $X$ be a simply-connected 4-manifold and $A$ an embedded surface in $X$ of genus $g_A$ and self-intersection $A^2=-N<0$. Suppose that the integral homology class $[A]$ is divisible by a prime power $q>1$. Then
\begin{equation}\label{ineq G-sig 2}
N\leq 2(b_2(X)-\sigma(X))+4g_A
\end{equation}
if $q=2$ and
\begin{equation}\label{ineq G-sig q0}
N\leq \frac{2q^2}{q^2-1}(b_2(X)-\sigma(X))+\frac{4q^2}{q^2-1}g_A
\end{equation}
if $q$ is an odd prime power. Hence for fixed genus $g_A$ there is an upper bound on $N$, depending only on $X$ and $q$. For an arbitrary odd prime power $q\geq 3$ we have
\begin{equation}\label{ineq G-sig q1}
N\leq \frac{9}{4}(b_2(X)-\sigma(X))+\frac{9}{2}g_A,
\end{equation}
independent of $q$.
\end{thm}
\begin{proof}
In \cite{R} it is shown using branched coverings over $X$, branched along $A$, that
\begin{equation*}
g_A\geq \Big\lvert\frac{\sigma(X)}{2} -\frac{A^2}{4}\Big\rvert -\frac{b_2(X)}{2}
\end{equation*}
if $q=2$ and
\begin{equation*}
g_A\geq \Big\lvert \frac{\sigma(X)}{2}-\frac{(q^2-1)A^2}{4q^2}\Big\rvert -\frac{b_2(X)}{2}
\end{equation*}
if $q$ is odd. This implies inequalities \eqref{ineq G-sig 2} and \eqref{ineq G-sig q0}. Suppose $q\geq 3$. With $b_2(X)\geq\sigma(X)$ we have
\begin{align*}
N&\leq\frac{2}{1-\tfrac{1}{q^2}}(b_2(X)-\sigma(X))+\frac{4}{1-\tfrac{1}{q^2}}g_A\\
&\leq \frac{9}{4}(b_2(X)-\sigma(X))+\frac{9}{2}g_A.
\end{align*}
This is inequality \eqref{ineq G-sig q1}.
\end{proof}
Note that the bound for $q$ odd is greater than the one for $q=2$ and tends to this one for large $q$. In any case $q^2$ has to divide $N$. Hence for a given prime power $q$ and genus $g_A$ there is only a small finite set of possible values for $N$. For completeness we also derive a bound using the $\frac{5}{4}$-conjecture. The inequality coming from this conjecture is more crucial in the case of a characteristic class in Section \ref{sect char}.
\begin{thm}\label{thm div}
Let $X$ be a simply-connected 4-manifold and $A$ an embedded surface in $X$ of genus $g_A$ and self-intersection $A^2=-N<0$. Suppose that the integral homology class $[A]$ is divisible by a prime power $q>1$ and let $[C]=\frac{1}{q}[A]$. Assume also one of the following:
\begin{itemize}
\item If $q$ is even, then the class $[C]$ is characteristic.
\item If $q$ is odd, then $X$ is spin.
\end{itemize}
Then
\begin{equation}\label{ineq div}
N\leq \frac{3q^2}{q^2-1}\left(\frac{4}{5}b_2(X)-\sigma(X)-\frac{8}{5q}\right)+\frac{24}{5}\frac{q}{q+1}g_A.
\end{equation}
Hence for fixed genus $g_A$ there is an upper bound on $N$, depending only on $X$ and $q$. If $X$ is spin and $q$ odd, then
\begin{equation}\label{ineq div2}
N\leq \frac{27}{40}(4b_2(X)-5\sigma(X))+\frac{24}{5}g_A,
\end{equation}
independent of $q$.
\end{thm}
\begin{proof}
We consider the branched covering $Y\rightarrow X$ as in Section \ref{sect facts} with $B=A$. Note that under our assumptions the manifold $Y$ is spin. Hence Furuta's theorem applies, which leads to inequality \eqref{ineq div}. Suppose that $X$ is spin and $q\geq 3$. An argument similar to the one above leads to inequality \eqref{ineq div2}. Here we have to use that $4b_2(X)-5\sigma(X)$ is non-negative according to the $\frac{5}{4}$-inequality applied to $X$. 
\end{proof}

\begin{ex}
Let $X$ be the $K3$ surface and $A$ a class divisible by two. With $b_2(X)=22$ and $\sigma(X)=-16$ we get from inequality \eqref{ineq G-sig 2}
\begin{equation*}
N\leq 76+4g_A.
\end{equation*}
In particular, for $g_A=0$, we see that there is an upper bound on $N$ of 76 for homology classes divisible by 2 that are represented by an embedded sphere. We can use inequality \eqref{ineq G-sig q1} to derive the following upper bound on $N$ if the class $A$ is divisible by an odd prime:
\begin{equation*}
N\leq 85.5 +4.5g_A.
\end{equation*}
\end{ex}

\section{The case of characteristic $A$}\label{sect char}

We begin with the following lemma:
\begin{lem}\label{KrMr}
Let $A$ be an embedded surface in $X$ of genus $g_A$ with self-intersection $A^2=-N<0$. Then the class $[B]=2[A]$ can be represented by a surface $B$ of genus $g_B=N-1+2g_A$.
\end{lem}
\begin{proof} The proof is elementary: Let $A'$ be a section of the normal disk bundle of $A$, considered as a tubular neighbourhood, having precisely $N$ transverse zeroes. Smoothing the intersections between $A'$ and $A$ results in the surface $B$. The claim also follows from the more general Lemma 7.7 in \cite{KrMr1} under switching the orientation of $X$.
\end{proof}
We can now prove the main theorem in this section.
\begin{thm}\label{thm char}
Let $X$ be a simply-connected 4-manifold and $A$ an embedded surface in $X$ of genus $g_A$ with self-intersection $A^2=-N<0$ such that the class $[A]$ is characteristic. Then
\begin{equation}\label{ineq char}
N\leq 4b_2(X)-5\sigma(X)-8+8g_A.
\end{equation}
Hence for fixed genus $g_A$ there is an upper bound on $N$, depending only on $X$.
\end{thm}
\begin{proof}
Let $[B]=2[A]$. According to Lemma \ref{KrMr}, the class $[B]$ is represented by a surface $B$ of genus $g_B=N-1+2g_A$. Let $Y\rightarrow X$ be the 2-fold branched covering, branched over $B$. Then $Y$ is spin and Furuta's theorem applies. We get:
\begin{equation*}
2b_2(X)+2g_B\geq \frac{5}{4}|2\sigma(X)+2N|+2\geq \frac{5}{4}(2\sigma(X)+2N)+2.
\end{equation*}
This implies inequality \eqref{ineq char}.
\end{proof}

\begin{rem}
Note that this idea only works for the 2-fold branched covering and not for higher degrees with the generalization of Lemma \ref{KrMr} due to the $q$-coefficients in the $\frac{5}{4}$-inequality. This is still the case if we replace the $\frac{5}{4}$-theorem by the $\frac{11}{8}$-conjecture. The argument does not work with the inequality $b_2(Y)\geq |\sigma(Y)|$ nor with the $G$-signature theorem.
\end{rem}

According to Kervaire and Milnor \cite{KeMi} a characteristic homology class $A$ represented by a sphere has to satisfy
\begin{equation*}
N=-A^2\equiv -\sigma(X) \bmod 16.
\end{equation*}
Together with the inequality in Theorem \ref{thm char} this can be used to restrict the possible self-intersection numbers of characteristic homology classes represented by spheres.
\begin{ex}
Let $X=\mathbb{CP}^2\#k{\overline{\mathbb{CP}}{}^{2}}$ with $k\geq 1$ and $A$ a characteristic homology class of self-intersection $A^2=-N<0$ represented by a sphere. By Theorem \ref{thm char}
\begin{equation*}
N\leq 9(k-1)
\end{equation*}
and according to Kervaire and Milnor
\begin{equation*}
N\equiv k-1\bmod 16.
\end{equation*}
Hence if $k=2$ then $N=1$, if $k=3$ then $N=2$ or $18$, and so on; compare with \cite{LiLi3}.
\end{ex}

\section{The case of indivisible $A$}

We consider the following conjecture, motivated by the results in Section \ref{sect div}.
\begin{conj}\label{conj 1}
Let $X$ be a simply-connected 4-manifold and $B$ an embedded surface in $X$ of genus $g_B$ with self-intersection $B^2=-M<0$ such that the class $[B]$ is divisible by $2$. Then
\begin{equation*}
M\leq c(X)+\kappa g_B,
\end{equation*}
where $c(X)$ is a number that depends only on $X$ and $\kappa$ is a number less than 4.
\end{conj}
Assuming that this is true we can derive a bound on $A^2=-N$ without any assumptions on the class $[A]$.
\begin{thm}
Assume that Conjecture \ref{conj 1} holds. Let $X$ be a simply-connected 4-manifold and $A$ an embedded surface in $X$ of genus $g_A$ with self-intersection number $A^2=-N<0$. Then
\begin{equation}\label{ineq conj}
N\leq \frac{1}{4-\kappa}(c(X)+\kappa(2g_A-1)).
\end{equation}
\end{thm}
\begin{proof}
The argument is a generalization of the proof of Theorem \ref{thm char}. Let $[B]=2[A]$. According to Lemma \ref{KrMr}, the class $[B]$ is represented by a surface $B$ of genus $g_B=N-1+2g_A$. We have $M=-B^2=4N$. Hence Conjecture \ref{conj 1} implies
\begin{equation*}
4N\leq c(X)+\kappa(N-1+2g_A)
\end{equation*}
This leads to inequality \eqref{ineq conj}.
\end{proof}
\begin{rem}
The inequality $b_2(Y)\geq |\sigma(Y)|$ and the inequality from the $G$-signature theorem have $\kappa=4$. The $\frac{5}{4}$-conjecture has $\kappa=\frac{16}{5}<4$, but only if $\frac{1}{2}[B]$ is characteristic. This leads to the same bound as in Theorem \ref{thm char}. We could also formulate a conjecture if the class $[B]$ is divisible by some other positive integer $d$. If $[B]=d[A]$ with $A^2=-N<0$ and $A$ is a surface of genus $g_A$, then $[B]$ is represented according to Lemma 7.7 in \cite{KrMr1} by a surface of genus
\begin{equation*}
g_B=\frac{1}{2}d(d-1)N+1-d+dg_A.
\end{equation*}
In the conjecture we have to replace $\kappa<4$ by $\kappa<\frac{2d}{d-1}$. However, for $d>2$, this conjecture tends to get unrealistic.
\end{rem}

\bibliographystyle{amsplain}

\end{document}